\documentclass[a4paper,10pt]{article}

\usepackage[utf8]{inputenc}
\usepackage{amssymb,amsthm,amsrefs,amsmath,epsfig,graphicx}
\usepackage[dvips]{color}

\DeclareMathOperator{\comp}{comp}
\DeclareMathOperator{\dist}{dist}
\DeclareMathOperator{\di}{d}
\DeclareMathOperator{\diam}{diam}
\DeclareMathOperator{\clos}{clos}

\newcommand{\Z}{\mathbb{Z}}
\newcommand{\R}{\mathbb{R}}

\newcommand{\K}{\mathcal{K}}
\newcommand{\F}{\mathcal{F}}
\newcommand{\KC}{\mathcal{C}}

\renewcommand{\epsilon}{\varepsilon}

\theoremstyle{plain}
\newtheorem{thm}{Theorem}[section]
\newtheorem{coro}[thm]{Corollary}

\newtheorem{prop}[thm]{Proposition}

\theoremstyle{definition}
\newtheorem{df}{Definition}[section]

\theoremstyle{remark}

\title{Lyapunov functions via Whitney's size functions}
\author{Alfonso Artigue}
\date{\today}

\begin{document}
\maketitle
\begin{abstract}
In this paper we present a 
technique for constructing Lypunov functions based on 
Whitney's size functions.
Applications to asymptotically stable equilibrium points, 
isolated sets, expansive homeomorphisms and 
continuum-wise expansive homeomorphisms are given. 

\end{abstract}

\section{Introduction}

In Dynamical Systems and Differential Equations it is important to determine the stability of trajectories
and a well known technique for this purpose is to find a Lyapunov function. 
In order to fix ideas consider a continuous flow $\phi\colon \R\times X\to X$ 
on a compact metric space $(X,\dist)$ with a singular (or equilibrium) point $p\in X$,
i.e., $\phi_t(p)=p$ for all $t\in\R$.
A Lyapunov function for $p$
is a continuous non-negative function that vanishes only at $p$ 
and strictly decreases along the orbits close to $p$. 
Recall that $p$ is \emph{stable} if for all $\epsilon>0$ there is $\delta>0$ such that if $\dist(x,p)<\delta$ then $\dist(\phi_t(x),p)<\epsilon$ for all $t\geq 0$. 
We say that $p$ is \emph{asymptotically stable} if it is stable and there is $\delta_0>0$ such that 
if $\dist(x,p)<\delta_0$ then $\phi_t(x)\to p$ as $t\to+\infty$. 
The existence of a Lyapunov function for an equilibrium point implies the asymptotic 
stability of the equilibrium point. 

A remarkable result, first proved by Massera in \cite{Massera49}, is the converse: every asymptotically stable singular point 
admits a Lyapunov function. 
Later, other authors obtained Lyapunov functions with different methods, 
see for example \cites{BaSz,Conley78}. 
In \cite{Hur} a generalization is proved in the context of arbitrary metric spaces.
The purpose of the present paper is to develop a different technique that allows us to 
construct Lyapunov functions for different dynamical systems as: isolated sets, expansive homeomorphisms and continuum-wise expansive homeomorphisms. 
Our techniques are based on the size function $\mu$ introduced by Whitney in \cite{Whitney33}.

In order to motivate our work let us show how to construct a Lyapunov function for an 
asymptotically stable singular point.
Denote by $\K(X)$ the set of non-empty compact subsets of $X$.
In the set $\K(X)$ we consider the Hausdorff distance 
$\dist_H$ making $(\K(X),\dist_H)$ a metric space.
Recall that 
\[
 \dist_H(A,B)=\inf\{\epsilon>0:A\subset B_\epsilon(B)\hbox{ and } B\subset B_\epsilon(A)\},
\]
where $B_\epsilon(C)=\cup_{x\in C} B_\epsilon(x)$ and $B_\epsilon(x)$ is the usual ball of radius $\epsilon$ 
centered at $x$.
See \cite{Nadler} for more on the Hausdorff metric.
A \emph{size function} is a continuous map $\mu\colon\K(X)\to\R$ satisfying:
\begin{enumerate}
 \item $\mu(A)\geq 0$ with equality if and only if $A$ has only one point,
 \item if $A\subset B$ and $A\neq B$ then $\mu(A)<\mu(B)$.
\end{enumerate}
In \cite{Whitney33} it is proved that size functions exists for every compact metric space. 

\begin{thm}
\label{estAs}
 If $\phi$ is a continuous flow on $X$ with an asymptotically stable singular point $p$ then 
 there are an open set $U$ containing $p$ and a continuous function $V\colon U\to \R$ satisfying: 
 \begin{enumerate}
  \item $V(x)\geq 0$ for all $x\in U$ with equality if and only if $x=p$ and
  \item if $t>0$ and $\{\phi_s(x): s\in[0,t]\}\subset U$ then $V(\phi_t(x))<V(x)$.
 \end{enumerate}
\end{thm}

\begin{proof}
 By the conditions on $p$ there are $\delta_0,\delta>0$ such that if $\dist(x,p)<\delta$ then 
 $\phi_t(x)\in B_{\delta_0}(p)$ for all $t\geq 0$ and $\phi_t(x)\to p$ as $t\to\infty$. 
 Define $U=B_\delta(p)$ and $V\colon U\to \R$ as
 \[
 V(x)=\mu(\{\phi_t(x):t\geq 0\}\cup\{p\})
 \]
 where $\mu$ is a size function.
 Since $\phi_t(x)\to p$ we have that 
 \begin{equation}\label{ecuO}
  O(x)=\{\phi_t(x):t\geq 0\}\cup\{p\}
 \end{equation}
 is a compact set for all $x\in U$. 
 Notice that if $t>0$ then $O(\phi_t(x))\subset O(x)$ and the inclusion is proper.
 Therefore, $V(\phi_t(x))<V(x)$ because $\mu$ is a size function. 
 Also notice that $V(p)=0$ and $V(x)>0$ if $x\neq p$. 
 In order to prove the continuity of $V$, 
 we will prove the continuity of $O\colon U\to K(X)$, the map defined by (\ref{ecuO}). 
 Since $\mu$ is continuous we will conclude the continuity of $V$. 
 
 Let us prove the continuity of $O$ at $x\in U$. 
 Take $\epsilon>0$. 
 By the asymptotic stability of $p$ there are $\rho,T>0$ such that 
 if $y\in B_\rho(x)$ then $\phi_t(y)\in B_{\epsilon/2}(p)$ for all $t\geq T$. 
 By the continuity of the flow, there is $r>0$ such that if 
 $y\in B_r(x)$ then $\dist(\phi_t(x),\phi_t(y))<\epsilon$ for all 
 $t\in[0,T]$. 
 Now it is easy to see that 
 if $y\in B_{\min\{\rho,r\}}(x)$ then 
 $\dist_H(O(x),O(y))<\epsilon$, proving the continuity of $O$ at $x$ and consequently the continuity of $V$.
\end{proof}

Let us recall that size functions can be easily defined. 
A variation of the construction given in \cite{Whitney33}, adapted for compact metric spaces, is the following. 
Let $q_1,q_2,q_3,\dots$ be a sequence dense in $X$. 
Define $\mu_i\colon \K(X)\to\R$ as 
\[
 \mu_i(A)=\max_{x\in A}\dist(q_i,x)-\min_{x\in A}\dist(q_i,x).
\]
The following formula defines a size function $\mu\colon \K(X)\to\R$
\[
 \mu(A)=\sum_{i=1}^\infty \frac{\mu_i(A)}{2^i},
\]
as proved in \cite{Whitney33}. 
In Section \ref{secLyapIso} we extend Theorem \ref{estAs} by
constructing a Lyapunov function for an isolated invariant sets.

For the study of expansive homeomorphisms (see Definition \ref{dfExp}) Lewowicz introduced in \cite{Lew} Lyapunov functions. 
He proved that expansiveness is equivalent with the existence of such function.
In Section \ref{secLyapHomeo} we give a different proof of this result by constructing a Lyapunov function 
defined for compact subsets of the space.
In \cite{Kato} Kato introduced another form of expansiveness called continuum-wise expansiveness (see Definition \ref{dfCwExp}). 
With our techniques we prove that continuum-wise expansiveness is equivalent with the existence 
of a Lyapunov function on continua subsets of the space.

\section{Lyapunov Functions for Isolated Sets}
\label{secLyapIso}

In this section we consider continuous flows on compact metric spaces. 
The purpose is to construct a Lyapunov for an isolated set of the flow using a size function. 
First we consider the case of an isolated set consisting of a point.

\subsection{Isolated Singularities}

Let $\phi$ be a continuous flow on a compact metric space $(X,\dist)$.
A point $p\in X$ is \emph{singular} for $\phi$ if $\phi_t(p)=p$ for all $t\in\R$.
A singular point $p\in X$ is \emph{isolated} if there is an open \emph{isolating neighborhood} $U$ of $p$ such that 
if $\phi_\R(x)\subset U$ then $x=p$. 

\begin{df}
\label{dfAdNei}An open set $U$ is an \emph{adapted neighborhood} of an isolated singular point $p\in U$ if 
for every orbit segment $l\subset\clos(U)$ with extreme points in $U$ it holds that 
$l\subset U$.
\end{df}

Given a set $A\subset X$ and $x\in A$ denote by $\comp_x(A)$ the connected component of $A$ that contains the point $x$.

\begin{prop}
Every isolated singular point has an adapted neighborhood.
\end{prop}

\begin{proof}
 Let $r>0$ be such that $\clos(B_r(p))$ is contained in an isolating neighborhood of $p$. 
 For $\rho\in (0,r)$ define the set
 \[
  U_\rho=\{x\in B_r(p):\comp_x(\phi_\R(x)\cap B_r(p))\cap B_\rho(p)\neq\emptyset\}.
 \]
By the continuity of the flow we have that $U_\rho$ is an open set for all $\rho\in(0,r)$. 
Let us prove that if $\rho$ is sufficiently small then $U_\rho$ is an adapted neighborhood. 
By contradiction, suppose that there are $\rho_n\to 0$, $a_n,b_n\in U_{\rho_n}$, $t_n\geq 0$ such that $b_n=\phi_{t_n}(a_n)$ and 
$l_n=\phi_{[0,t_n]}(a_n)\subset \clos(U_{\rho_m})$ but $l_n$ is not contained in $U_{\rho_n}$. 
Then there is $s_n\in(0,t_n)$ such that $\phi_{s_n}(a_n)\in\partial B_r(p)$. 
Also, there must be $u_n<0$ and $v_n>0$ such that $\phi_{u_n}(a_n),\phi_{v_n}(b_n)\in B_{\rho_n}(p)$. 
But a limit point of $\phi_{s_n}(a_n)$ contradicts that $\clos(B_r(p))$ is contained in an insolating neighborhood of $p$.
\end{proof}

Fix an isolated point $p$ with an adapted neighborhood $U$.
Consider the sets
\[
\begin{array}{l}
\displaystyle W^s_U(p)=\{x\in U : \lim_{t\to+\infty}\phi_t(x)= p\hbox{ and } \phi_{\R^+}(x)\subset U\},\\
\displaystyle W^u_U(p)=\{x\in U : \lim_{t\to-\infty}\phi_t(x)= p\hbox{ and } \phi_{\R^-}(x)\subset U\},\\.
\end{array}
\]

For $x\in U$ define the orbit segments
\[
\begin{array}{l}
O^+_U(x)=\comp_x(U\cap \phi_{[0,+\infty)}(x)),\\
O^-_U(x)=\comp_x(U\cap \phi_{(-\infty,0]}(x)).
\end{array}
\]
Define $C=X\setminus U$ and let $V_p^+,V_p^-\colon U\to \K(X)$ be defined as 
\[
\left\{
\begin{array}{l}
 V_p^+(x)=\clos(O^+_U(x)\cup W^u_U(p))\cup C,\\
 V_p^-(x)=\clos(O^-_U(x)\cup W^s_U(p))\cup C.
\end{array}
\right.
\]

\begin{df}
A \emph{Lyapunov function} for an isolated point $p$ is a continuous map $V\colon U\to\R$ 
defined in a neighborhood of $p$ such that if $t>0$ and $\phi_{[0,t]}(x)\subset U\setminus\{p\}$ 
then 
$V(x)>V(\phi_t(x))$.
\end{df}

\begin{thm}
\label{LyapSing}
If $p$ is an isolated point and $U$ is an adapted neighborhood of $p$
then the maps $V_p^+$ and $V_p^-$ are continuous 
in $U$. 
If in addition, $\mu$ is a size function on $\K(X)$
then $V\colon U\to\R$ defined as 
\[
 V(x)=\mu(V_p^+(x))-\mu(V_p^-(x))
\]
is a Lyapunov function for $p$.
\end{thm}

\begin{proof}
Let us prove the continuity of $V_p^+$ by contradiction.
Assume that $x_n\to x\in U$ and $V^+_p(x_n)\to K$ with the Hausdorff distance but $K\neq V^+_p(x)$. 
By definitions we have that 
 \begin{equation}\label{incluComp}
  \clos(W^u_U(p))\cup C\subset K\cap V^+_p(x).  
 \end{equation}
 Recall that $C$ was defined as the complement of $U$ in $X$.
 Take a point $y\in K\setminus V^+_p(x) \cup V^+_p(x) \setminus K$.
 By the inclusion (\ref{incluComp}) we know that $y\notin \clos(W^u_U(p))\cup C$. 
We divide the proof in two cases.

 \emph{Case 1}. Suppose first that $y\in K\setminus V^+_U(x)$. 
 Since $y\in K$ there is a sequence $t_n\geq 0$ such that 
 $\phi_{t_n}(x_n)\to y$ and $\phi_{[0,t_n]}(x_n)\subset U$. 
 If $t_n\to\infty$ then $x\in W^s_U(p)$. 
 Consequently, $y\in W^u_U(p)$, which is a contradiction. 
 Therefore $t_n$ is bounded. Without loss of generality assume that $t_n\to\tau\geq 0$ and then $\phi_\tau(x)=y$. 
 Thus $\phi_{[0,\tau]}(x)\subset\clos(U)$.
 Since $y\notin C$ we have that $y\in U$. 
 Now, since $U$ is an adapted neighborhood we conclude that $\phi_{[0,\tau]}(x)\subset U$ and 
 then $y\in O^+(x)\subset V^+_p(x)$.
 This contradiction finishes this case.
 
 \emph{Case 2}. Now assume that $y\in V^+_p(x)\setminus K$. In this case we have that 
 $y=\phi_s(x)$ for some $s\geq 0$ and $\phi_{[0,s]}(x)\subset U$. 
 Then $\phi_s(x_n)\to y$ and $y\in K$. 
 This contradiction proves that $V^+_p$ is continuous in $U$. 

 The continuity of $V^-_p$ is proved in a similar way. 
 Let us show that $V$ is a Lyapunov function for $p$. 
 The continuity of $V$ in $U$ follows by the continuity of $V^+_p$, $V^-_p$ and the size function $\mu$.

Now take $x\notin U\setminus \{p\}$. 
 We will show that $V$ decreases along the orbit segment of $x$ contained in $U$. 
 Notice that for all $t>0$, $O^+_U(\phi_t(x))\subset O^+_U(x)$ if $\phi_{[0,t]}(x)\subset U$. 
 Therefore $V^+_p(\phi_t(x))\leq V^+_p(O^+(x))$. 
 The equality can only hold if $x\in W^u_U(p)$. 
 But in this case we have that $x\notin W^s_U(p)$ 
 because $W^u_U(p)\cap W^s_U(p)=\{p\}$. 
 Then $V^-_p(\phi_t(x))>V^-_p(x)$. 
 Therefore, $V(\phi_t(x))<V(x)$ and $V$ is a Lyapunov function for $p$.
\end{proof}

\subsection{Isolated Sets}

Let $\phi\colon\R\times X\to X$ be a continuous flow on a compact metric space $X$. 
Consider a $\phi$-invariant set $\Lambda\subset X$, i.e., $\phi_t(\Lambda)=\Lambda$ for all $t\in\R$. 
We say that $\Lambda$ is an \emph{isolated set} with \emph{isolating neighborhood} $U$ if 
$\phi_\R(x)\subset U$ implies $x\in\Lambda$. 
\begin{df}
A \emph{Lyapunov function} for an isolated set $\Lambda$ is a continuous function $V\colon U\to \R$ 
defined on an open set $U$ containing $\Lambda$ such that:
\begin{enumerate}
 \item $V(x)=0$ if and only if $x\in\Lambda$,
 \item if $\phi_{[0,t]}(x)\subset U\setminus\Lambda$ then $V(x)>V(\phi_t(x))$.
\end{enumerate}
\end{df}

Let us show how the construction of a Lyapunov function for an isolated set 
can be reduced to the case of an isolated singular point. 

\begin{thm}
\label{LyapIso}
 Every isolated set admits a Lyapunov function.
\end{thm}

\begin{proof}
Consider the set $Y=(X\setminus \Lambda)\cup\{\Lambda\}$. 
On $Y$ define the distance $\di$ as 
\[
 \di(x,y)=\min\{\dist(x,y),\dist(x,\Lambda)+\dist(y,\Lambda)\}.
\]
It is easy to see that $(Y,\di)$ is a compact metric space. 
Also, the flow $\phi$ induces naturally a flow $\phi'$ on $Y$ 
with $\Lambda$ as an isolated singular point. 
Consider from Theorem \ref{LyapSing} a Lyapunov function for $\Lambda$ as an isolated singular point of $\phi'$. 
This function naturally defines a Lyapunov function for $\Lambda$ as an isolated set of $\phi$.
\end{proof}

\section{Applications to homeomorphisms}
\label{secLyapHomeo}
Let $f\colon X\to X$ be a homeomorphism of a compact metric space $(X,\dist)$. 
An $f$ invariant set $\Lambda$ is \emph{isolated} if there is an open neighborhood $U$ of $\Lambda$ such that
$f^n(x)\in U$ for all $n\in\Z$ implies that $x\in \Lambda$. 

\begin{thm}
\label{teoLyapDisc}
Every isolated set $\Lambda$ for a homeomorphism $f$ admits a Lyapunov function, 
that is, a continuous map $V\colon U\subset X\to\R$ defined on a neighborhood of $\Lambda$ 
such that:
\begin{enumerate}
 \item $V(x)=0$ if and only if $x\in\Lambda$, 
 \item $V(x)>V(f(x))$ if $x,f(x)\in U\setminus\Lambda$.
\end{enumerate}
\end{thm}

\begin{proof}
Consider $\phi\colon \R\times X_f\to X_f$ 
the suspension of $f$.
Consider $i\colon X\to X_f$ a homeomorphism onto its image such that 
$i(X)$ is a global cross section of $\phi$.
It is easy to see that $\Lambda$ is an isolated set for $f$ if and only 
$\Lambda_f=\phi_\R(i(\Lambda))$ is an isolated set for $\phi$.
Now consider a Lyapunov function $V'$ for $\Lambda_f$. 
A Lyapunov function for $f$ can be defined by $V(x)=V'(i(x))$.
\end{proof}

\begin{df}
\label{dfExp}
A homeomorphism $f\colon X\to X$ of a compact metric space is \emph{expansive} if 
there is $\alpha>0$ (an \emph{expansive constant}) such that if $x\neq y$ then there is $n\in\Z$ such that $\dist(f^n(x),f^n(y))>\alpha$. 
\end{df}

Recall that $\K(X)$ denotes the compact metric space of compact subsets of $X$ with the Hausdorff metric. 
Denote by $\F_1=\{A\in\K(X):|A|=1\}$ where $|A|$ denotes the cardinality of $A$.
Given a homeomorphism $f\colon X\to X$ define the homeomorphism 
$f'\colon \K(X)\to \K(X)$ as $f'(A)=\{f(x): x\in A\}$.
Notice that $\F_1$ is invariant under $f'$.

\begin{coro}
\label{LyapExp}
For a homeomorphism $f\colon X\to X$ the following statements are equivalent:
\begin{enumerate}
 \item $f$ is an expansive homeomorphism,
 \item $\F_1$ is an isolated set for $f'$, 
 \item there is a continuous function $V\colon U\subset \K(X)\to\R$ defined on a neighborhood of $\F_1$ 
 such that $V(A)=0$ if and only if $A\in \F_1$ and
 $V(A)>V(f'(A))$ if $A,f'(A)\in U\setminus \F_1$.
\end{enumerate}
\end{coro}

\begin{proof}
($1\to 2$).
Let $\delta$ be an expansive constant and 
define $$U=\{A\in\K(X):\diam(A)<\delta\}.$$ 
It is easy to see that $U$ is an isolating neighborhood of $\F_1$.

($2\to 3$). It follows by Theorem \ref{teoLyapDisc}.

($3\to 1$). Take $\delta>0$ such that if 
$\dist(x,y)\leq\delta$ then $\{x,y\}\in U$. 
Let us prove that $\delta$ is an expansive constant for $f$.
Assume by contradiction that $\dist(f^n(x),f^n(y))\leq\delta$ for all $n\in\Z$ 
and $x\neq y$. 
Define $A=\{x,y\}$. 
We have that $V(f'^n(A))$ is a decreasing sequence. 
Without loss of generality assume that $V(A)<0$.
Suppose that $f'^n(A)$ accumulates in $B$. 
Now it is easy to see that $B\in U\setminus \F_1$ and also $V(B)=V(f'(B))$. 
This contradiction proves the theorem.
\end{proof}

Recall that a \emph{continuum} is a compact connected set.
Denote by $\KC(X)=\{C\in\K(X):C\hbox{ is connected} \}$ the space of continua of $X$.

\begin{df}
\label{dfCwExp}
A homeomorphism $f\colon X\to X$ is \emph{continuum-wise expansive} if there is 
$\delta>0$ such that if $C\in\KC(X)$ and 
$\diam(f^n(C))\leq\delta$ for all $n\in \Z$ then $C\in \F_1$. 
\end{df}

A \emph{Lyapunov function} for a continuum-wise expansive homeomorphism is a continuous function $V\colon U\subset \KC(X)\to \R$ defined on a neighborhood 
of $\F_1(X)$ in $\KC(X)$ such that 
$V(\{x\})=0$ for all $x\in X$ and 
$V(f(C))<V(C)$ if $C\notin \F_1$ and $C,f(C)\in U$.

\begin{coro}
For a homeomorphism $f\colon X\to X$ the following statements are equivalent:
\begin{enumerate}
 \item $f$ is a continuum-wise expansive homeomorphism,
 \item $\F_1$ is an isolated set for $f'\colon\KC(X)\to\KC(X)$, 
 \item there is a continuous function $V\colon U\subset \KC(X)\to\R$ defined on an open set $U\subset\KC(X)$ containing $\F_1$
 such that $V(A)=0$ if and only if $A\in \F_1$ and $V(A)>V(f'(A))$ if $A,f'(A)\in U\setminus \F_1$.
\end{enumerate}
\end{coro}

\begin{proof}
The proof is similar to the proof of Corollary \ref{LyapExp}.
\end{proof}

\end{document}